\documentclass[a4paper]{amsart}

\usepackage{amsmath}%
\usepackage{amsfonts}%
\usepackage{amssymb}%
\usepackage{amsrefs}
\newtheorem{theorem}{Theorem}
\theoremstyle{plain}

\newtheorem{corollary}{Corollary}

\newtheorem{lemma}{Lemma}

\newtheorem{proposition}{Proposition}

\numberwithin{equation}{section}

\theoremstyle{definition}
\newtheorem{definition}{Definition}

\newcommand{\RR}{\mathbb{R}}

\newcommand{\rmM}{\mathrm{M}}

\newcommand{\calK}{\mathcal{K}}
\newcommand{\calP}{\mathcal{P}}
\newcommand{\calS}{\mathcal{S}}

\DeclareMathOperator{\vol}{vol}

\DeclareMathOperator{\conv}{conv}

\DeclareMathOperator{\spn}{span}

\begin{document}
\title[Minkowski valuations]{$GL(n)$ equivariant Minkowski valuations}

\author{Thomas Wannerer}
\address{\begin{flushleft}
Vienna University of Technology \\
Institute of Discrete Mathematics and Geometry \\
Wiedner Hauptstra\ss e 8--10/1046 \\
A--1040 Vienna, Austria
\end{flushleft}}
\email{thomas.wannerer@tuwien.ac.at}

\subjclass[2010]{Primary 52A20; Secondary 52B45, 52A40}

\begin{abstract} A classification of all continuous $GL(n)$ equivariant Minkowski valuations on convex bodies in $\RR^n$ is established. Together with recent results of F.E.~ Schuster and the author, this article therefore completes the description of all continuous 
$GL(n)$ intertwining Minkowski valuations.
\end{abstract}
\maketitle
\section{Introduction}

The centroid body is a classical notion in the affine geometry of convex bodies, which has attracted increased attention in recent years, cf.\ \cites{berck10,campi-gronchi,fleury,gardner-book, gardner-gianno,grinberg,haberl08b,lutwak1,lutwak2,paouris, yaskin}. For an origin-symmetric convex body $K$, the boundary of the centroid body of $K$ is the locus of the centroids of halves of $K$ formed by slicing $K$ by hyperplanes through the origin. 

The Busemann-Petty centroid inequality \cite{petty61} states that among convex bodies of given volume, precisely the centroid bodies of centered ellipsoids have minimal volume. This important inequality could be extended to the $L_p$-Brunn-
Minkowski theory and recently to the Orlicz Brunn-Minkowski theory (see \cites{lz,lyz,lyz00duke, lyz10jdg}).

The difference body of a convex body is the Minkowski sum of $K$ and the reflection of $K$ at the origin, $-K$. The operation which defines the difference body is up to dilation the central symmetrization of $K$, which is a basic operation in geometry with numerous applications in different areas, cf.\ \cite{gardner-book}.
 
The underlying reason for the importance of centroid and difference bodies in geometry has only recently been demonstrated by Ludwig \cite{ Ludwig:Minkowski}: The difference and the centroid body operator are basically the only continuous $GL(n)$ \emph{equivariant} Minkowski valuations on the set of convex bodies containing the origin. In addition, the difference body operator was characterized as the only continuous Minkowski valuation on the set of all convex bodies which is $GL(n)$ equivariant \emph{and} translation invariant. 

In this article we establish a complete classification of all continuous and $GL(n)$ equivariant Minkowski valuations, without any further assumption on their behavior under translations or any restrictions on their domain.  In the case of $1$-homogeneous, $GL(n)$ equivariant valuations the only additional operator that appears is the convex hull of $K$ and the origin. For $(n+1)$-homogeneous, $GL(n)$ equivariant Minkowski valuations, however, a new operator emerges.

Let $\calK^n$ denote the space of convex bodies (compact
convex sets) in $\RR^n$, $n \geq 3$, endowed with the
Hausdorff metric and let $\calK_0^n$ denote the subset of
$\calK^n$ of bodies containing the origin. A convex body $K$
is uniquely determined by its support function
$h(K,x)=\max\{\left\langle x,y\right\rangle : y \in K\}$, $x \in \RR^n$.  Here $\left\langle x,y\right\rangle$ denotes the standard inner product of $x$ and $y$.
\pagebreak
\begin{definition}
Let $\calS\subset\calK^n$.   A map $\Phi: \calS
\rightarrow \calK^n$ is called a \emph{Minkowski valuation} if
$$\Phi K + \Phi L = \Phi(K \cup L) + \Phi(K \cap L)$$
whenever $K, L, K\cap L,  K \cup L \in \calS$. Here addition on
$\calK^n$ is Minkowski addition. 
The map $\Phi$ is called $GL(n)$ \emph{equivariant} if there
exists  $q \in \RR$ such that
$$\Phi(\phi K) = |\det \phi|^q \phi\Phi K$$
whenever $\phi \in GL(n)$ and $K, \phi K \in \calS$.
A Minkowski valuation $\Phi$ is called \emph{continuous} if $\Phi: \calS\rightarrow\calK^n$ is continuous with respect to the topology 
induced by the Hausdorff metric.
\end{definition}

We refer the reader to the surveys \cite{McMullen93} and \cite{M-S} for information on the classical theory of valuations and to \cites{Alesker99,Alesker01, bernig,bernig-broecker, bernig-fu, fu, ludwig-reitzner} for some of the more recent results on scalar valued valuations. First results on Minkowski valuations which are rotation equivariant were obtained by Schneider \cite{schneider74}
in the 1970s (see \cites{kiderlen05, schnschu, Schu06a, Schu09} for recent extensions of these results). The starting point for the systematic study of convex and star body valued valuations (see \cites{haberl08,haberl09,hab-lud,Ludwig:matrix,Ludwig06,lyz,Schu09}) were two highly influential articles by Ludwig \cites{ludwig02, Ludwig:Minkowski}. A new aspect in this area of research explores the connections between these valuations and affine isoperimetric inequalities, cf. \cites{habschu09, H-Sc-2}.

The \emph{moment body} $\rmM K$ of a convex body $K\in\calK^n$ is given by
$$h(\rmM K,x)=\int_K|\!\left\langle x,y\right\rangle\!|\; dy.$$
For $K\in\calK^n$ with nonempty interior the \emph{centroid body} $\Gamma K$ is defined by
$$\Gamma K=\frac{1}{\vol_n(K)} \rmM K.$$
Note that the convex bodies $\Gamma K$ and $\rmM K$ are just dilates. However, the moment body operator $\rmM$ is a valuation, while $\Gamma$ is not. It is not difficult to see that the moment body operator is continuous  and a change of variables in the integral shows that for any $\phi\in GL(n)$
$$\rmM(\phi K)=|\det\phi| \phi \rmM K.$$ 
The \emph{moment vector} $m(K)$ of a convex body $K$ is given by
$$m(K)=\int_K x\; dx.$$
The map $m:\calK^n\rightarrow\calK^n$ is another example of a $GL(n)$ equivariant continuous Minkowski valuation. Note that up to normalization, $m(K)$ is just the centroid of $K$. 

The following theorem was obtained in \cite{Ludwig:Minkowski}.

\begin{theorem}\label{thm:1}  A map $\Phi:\calK^n_0\rightarrow\calK^n$ is a $GL(n)$ equivariant continuous Minkowski valuation if and only if either there are constants $a_1, a_2\geq0$ such that
$$\Phi K=a_1K+a_2(-K)$$
for every $K\in\calK_0^n$ or there are constants $a_1\in\RR$ and $a_2\geq 0$ such that 
$$\Phi K = a_1 m(K)+a_2 \rmM K$$
for every $K\in\calK_0^n$.
\end{theorem}

If we consider Minkowski valuations which are defined on the set of \emph{all} convex bodies, then there are more $GL(n)$ equivariant Minkowski 
valuations than just the ones described in the above theorem. Consider for example the map of $\calK^n$ into $\calK^n$ which
sends $K$ to $K_o=\conv(\{0\}\cup K)$, the convex hull of $K$ and the origin. It is easy to show that this map is also a $GL(n)$ equivariant continuous Minkowski valuation. Further examples of $GL(n)$ equivariant continuous Minkowski valuations are given by similar modifications of the definitions of the moment vector and the 
moment body operator: The maps $M_*:\calK^n\rightarrow\calK^n$ defined by 
$$h(\rmM_*K,x)=\int_{K_o\setminus K} |\!\left\langle x,y\right\rangle\!|\; dy$$
and $m_*:\calK^n\rightarrow\calK^n$ defined by
$$m_*(K)=\int_{K_o\setminus K} x \; dx$$
are easily seen to be $GL(n)$ equivariant continuous Minkowski valuations.

The main goal of this article is to show that all $GL(n)$ equivariant continuous Minkowski valuations are given by linear combinations of 
$K$, $-K$, $K_o$, $-K_o$ and $m$, $\rmM$, $m_*$, $\rmM_*$, respectively.

\begin{theorem}\label{thm:2} A map $\Phi:\calK^n\rightarrow\calK^n$ is a $GL(n)$ equivariant continuous Minkowski valuation if and only if either there are constants $a_1, a_2, a_3, a_4\geq0$ such that
$$\Phi K=a_1K+a_2(-K)+a_3K_o+a_4(-K_o)$$
for every $K\in\calK^n$ or there are constants $a_1,a_2\in\RR$ and $a_3,a_4\geq 0$ such that 
$$\Phi K = a_1 m(K)+a_2 m_*(K)+a_3\rmM K+a_4\rmM_* K$$
for every $K\in\calK^n$.
\end{theorem}

Let $\calS\subset\calK^n$.   A map $\Phi: \calS\rightarrow \calK^n$ is called $GL(n)$ \emph{contravariant} if there exists  $q \in \RR$ such that
$$\Phi(\phi K) = |\det \phi|^q \phi^{-t}\Phi K$$
whenever $\phi \in GL(n)$ and $K, \phi K \in \calS$. Here $\phi^{-t}$ denotes the inverse of the transpose of $\phi$. The \emph{projection body} of a convex body $K$ is the unique convex body $\Pi K$ with the property that
$$h(\Pi K,u)=\vol_{n-1}(K|u^\perp)\qquad \text{for}\ u\in S^{n-1},$$
where $\vol_{n-1}(K|u^\perp)$ denotes the $(n-1)$-dimensional volume of the orthogonal projection of $K$ on the subspace orthogonal to $u$.  The operator $\Pi$ is a $GL(n)$ contravariant continuous Minkowski valuation, see for example \cite{gardner-book}.

Recently, continuous and $GL(n)$ contravariant Minkowski valuations have been classified by F.E.~ Schuster and the author. In \cite{SchuWann10} it was shown that every $GL(n)$ contravariant continuous Minkowski valuation is a linear combination of $\Pi$ and
$\Pi_o$, where $\Pi_oK:=\Pi(K_o)$. In view of this result, one might think that every $GL(n)$ equivariant continuous Minkowski
 valuation which is homogeneous of degree $n+1$ is a linear combination of $m$, $\rmM$, $m_o$ and $\rmM_o$, where $m_o(K):=m(K_o)$ and $\rmM_o(K):=\rmM(K_o)$. This, however, is false, 
since the family of Minkowski valuations obtained from linear combinations of the form
$$\Phi K = a_1 m(K)+a_2 m_o(K)+a_3\rmM K+a_4\rmM_o K,\quad a_1,a_2\in\RR,\ a_3,a_4\geq0,$$
is a proper subfamily of Minkowski valuations of the form 
$$\Phi K = a_1 m(K)+a_2 m_*(K)+a_3\rmM K+a_4\rmM_* K,\quad a_1,a_2\in\RR,\ a_3,a_4\geq0.$$
Indeed, consider for example $\rmM_*$. As can be seen from the definition, $\rmM_*K=\{0\}$ for every $K\in\calK^n$ containing the origin. On the other hand, if $a_1 m+a_2 m_o+a_3\rmM +a_4 \rmM_o$ vanishes on all convex bodies containing the origin, then it must in particular vanish on balls centered at the origin. Hence $a_3=a_4=0$ and we conclude that one cannot express $\rmM_*$ as a linear combination of $m$, $m_o$, $\rmM$ and $\rmM_o$.

Recall that a map $\Phi:\calS\rightarrow\calK^n$, $\calS\subset \calK^n$, is said to be (positively) \emph{homogeneous} of degree $r$ if 
$$\Phi(\lambda K)=\lambda^r \Phi K$$
whenever $\lambda>0$ and $K, \lambda K\in \calS$. Denote by $B^n$ the $n$-dimensional Euclidean unit ball. The Busemann-Petty centroid inequality states that for any $K\in\calK^n$ with non-empty interior
$$\vol_n(\rmM K)\vol_n(B^n)^{n+1}\geq \vol_n(\rmM B^n) \vol_n(K)^{n+1}.$$
For $K\in\calK^n$ with non-empty interior, equality holds if and only if $K$ is a centered ellipsoid. Let $\Phi$ be as in Theorem \ref{thm:2}. If $\Phi$ is homogenous of degree $n-1$, Theorem 2 yields $\vol_n(\Phi K)\geq \vol_n(a_3\rmM K)$ with some nonnegative constant $a_3$. Since $\Phi B^n=a_3\rmM B^n$ we get
$$\vol_n(\Phi K)\geq \frac{\vol_n(\Phi B^n)}{\vol_n(\rmM B^n)} \vol_n(\rmM K).$$
 Hence we obtain the following generalization of the Busemann-Petty inequality. 

\begin{corollary}
	Suppose $\Phi:\calK^n\rightarrow\calK^n$ is a $GL(n)$ equivariant continuous Minkowski valuation which is homogeneous of degree $n+1$. Then 
$$\vol_n(\Phi K)\vol_n(B^n)^{n+1}\geq \vol_n(\Phi B^n)\vol_n(K)^{n+1}$$
for every $K\in\calK^n$. For $K\in\calK^n$ with non-empty interior, equality holds if and only if $K$ is a centered ellipsoid.
\end{corollary}

\section{Auxiliary results}

In this section we show that any $SL(n)$ equivariant continuous Minkowski valuation which is homogeneous of degree $r\neq 1,n+1$ is trivial. 
The proof is similar to the $SL(n)$ contravariant case treated in \cite{SchuWann10} and is based on ideas and techniques developed by Ludwig
in \cite{Ludwig:Minkowski}. We also state for later reference some facts and results on convex bodies and Minkowski valuations.

 A map $\Phi:\calS\rightarrow\calK^n$, $\calS\subset \calK^n$, is called $SL(n)$ \emph{equivariant} if 
\begin{equation}\label{eq1}\Phi(\phi K)=\phi\Phi K\end{equation}
whenever $\phi\in SL(n)$ and $K, \phi K\in \calS$. 
Obviously, every $GL(n)$ equivariant map is $SL(n)$ equivariant and homogeneous. The following result is therefore a stronger version of Theorem \ref{thm:1} and can be found in \cite{Ludwig:Minkowski}. 

\begin{theorem} \label{thm:3} A map $\Phi:\calK^n_0\rightarrow\calK^n$ is an $SL(n)$ equivariant, homogeneous continuous Minkowski valuation if and only if either there are constants $a_1, a_2\geq0$ such that
$$\Phi K=a_1K+a_2(-K)$$
for every $K\in\calK^n_0$ or there are constants $a_1\in\RR$ and $a_2\geq 0$ such that 
$$\Phi K = a_1 m(K)+a_2\rmM K$$
for every $K\in\calK^n_0$.
\end{theorem}

We remark that if $\Phi$ is $SL(n)$ equivariant and homogeneous of degree $r$, then we have 
\begin{equation}\label{eq:1}\Phi(\phi K)=(\det \phi)^{(r-1)/n}\phi \Phi K\end{equation}
for every $K\in\calK^n$ and $\phi\in GL(n)$ with $\det\phi>0$.

There is a one-to-one correspondence between convex bodies and subadditive, positively homogeneous functions on $\RR^n$, given by the identification of a convex body $K\in\calK^n$ with its support function $h(K,\cdot)$ (for this fact and for more information on convex bodies we refer the reader to \cite{schneider_book}).  If $\phi \in GL(n)$ and $K \in \mathcal{K}^n$, then 
\begin{equation} \label{supptrans}
h(\phi K,x) = h(K,\phi^{t}x)
\end{equation}
for every $x \in \mathbb{R}^n$. For $K\in\calK^n$ denote by $I_K$ the subgroup of all $\phi\in SO(n)$ such that $\phi K= K$. Let $\Phi: \calK^n\rightarrow \calK^n$ be an $SL(n)$ equivariant map. The $SL(n)$-equivariance of $\Phi$ together with (\ref{supptrans}) implies that 
\begin{equation} \label{invar}
	h(\Phi K, \phi x)=h(\Phi K,x) \qquad\text{for any}\ \phi\in I_K
\end{equation}
and $x\in\RR^n$.

We start by gathering information about $\Phi K$ for convex bodies $K$ contained in hyperplanes. We denote by $\spn A$ the linear hull of 
$A\subset\RR^n$.
\begin{lemma}\label{lem:1}
Suppose $\Phi:\mathcal{K}^n\rightarrow\mathcal{K}^n$ is a Minkowski valuation which is $SL(n)$ equivariant and homogeneous of degree $r$. Then the following holds. 
\begin{itemize}
	\item[(i)] 
$\Phi K\subset\spn K$.
	\item[(ii)]
If $r\neq 1$ and $\spn K\neq \RR^n$, then $\Phi K=\{0\}$. 
\end{itemize}
\end{lemma}
\begin{proof}
Fix $1\leq k\leq n-1$ and let $\phi\in SL(n)$ be the matrix defined by
$$\phi=\begin{pmatrix} I_k & B\\ 0 & A\end{pmatrix},$$
where $I_k$ is the $k\times k$ identity matrix, $0$ is the $(n-k)\times k$ null matrix, $B$ is an arbitrary $k\times (n-k)$ matrix, and 
$A\in SL(n-k)$. Suppose  $K\in\calK^n$ and $\spn K =\spn\{e_1,\ldots,e_k\}$, where $\{e_1,\ldots,e_n\}$ is the standard
 orthonormal basis of $\RR^n$. Clearly, $\phi$ leaves $K$ fixed, i.e.,
\begin{equation}\label{eq:2}
 \phi K=K.
\end{equation}
Choose $x\in\Phi K$ and write $x=x'+x''$ with $x'\in\spn\{e_1,\ldots,e_k\}$ and $x''\in\spn\{e_{n-k},\ldots,e_n\}$. Since $\Phi$ is 
$SL(n)$ equivariant, it follows from (\ref{eq1}) and (\ref{eq:2}) that 
$$\phi x= (x'+Bx'')+Ax''\in \Phi K.$$
Since $B$ is an arbitrary $k\times (n-k)$ matrix and $\Phi K$ is bounded, this implies $x''=0$. Thus, 
$\Phi K\subset \spn\{e_1,\ldots,e_k\}=\spn K$. Assertion (i) now follows from (\ref{eq1}).

To prove (ii), assume $K\subset \spn\{e_1,\ldots,e_{n-1}\}$ and $r\neq 1$. Put
$$\psi=\begin{pmatrix} I_{n-1} &0\\ 0 & s\end{pmatrix}$$
for $s>0$. Then $\det\psi>0$ and by (\ref{eq:1})
$$\Phi K=\Phi(\psi K)=s^{(r-1)/n} \Phi K.$$
Since this holds for every $s>0$ and $\Phi K$ is bounded, we must have $\Phi K=\{0\}$. As before, (ii) now follows from (\ref{eq1}).
\end{proof}

The next lemma generalizes a result of Ludwig, \cite{Ludwig:Minkowski}*{Lemma 2}. We denote by $\calP^n\subset \calK^n$ the set of all convex polytopes in $\RR^n$. 
\begin{lemma}\label{lem:eq} Suppose $\Phi_1,\Phi_2:\calP^n\rightarrow\calK^n$ are Minkowski valuations such that $\Phi_1\{0\}=\Phi_2\{0\}$, $\Phi_1 T=\Phi_2 T$ for every $n$-dimensional simplex $T$ having one vertex at the origin, and $\Phi_1 S=\Phi_2 S$ for every $(n-1)$-dimensional simplex $S$ not contained in a hyperplane through the origin. Then
$$\Phi_1=\Phi_2.$$
\end{lemma}
\begin{proof}
	Let $\mu:\calP^n\rightarrow\RR$ be a real valued valuation such that $\mu(\{0\})=0$, $\mu(T)=0$ for every $n$-dimensional simplex $T$ having one vertex at the origin, and $\mu(S)=0$ for every $(n-1)$-dimensional simplex $S$ not contained in a hyperplane through the origin. We shall show by induction on $n$ that $\mu=0$ on $\calP^n$. For $n=1$ the claim is obvious. Suppose now $n>1$ and that the statement is true for $n-1$. Let $T'$ be an $(n-1)$-dimensional simplex having one vertex at the origin. There exist an $n$-dimensional simplex $T$ and a hyperplane $H$  such that $H$ dissects $T$ into two simplices $T^+$ and $T^-$, each having one vertex at the origin, and $H\cap T=T'$. Since $\mu$ is a valuation, it follows that 
$$\mu(T)+\mu(T')=\mu(T^+)+\mu(T^-).$$
From $\mu(T)=\mu(T^+)=\mu(T^-)=0$ we conclude that $\mu(T')=0$. Hence $\mu$ vanishes on $(n-1)$-dimensional simplices having one vertex at the origin. Similarly it can be shown that $\mu$ vanishes on $(n-2)$-dimensional simplices which are not contained in $(n-2)$-dimensional subspaces. By the induction hypothesis, $\mu(P)=0$ for every polytope $P$ contained in a hyperplane through the origin. 

Suppose $P$ is an $(n-1)$-dimensional polytope not contained in a hyperplane through the origin. We write $P=S_1\cup\cdots\cup S_m$, where the $S_i$ 
are $(n-1)$-dimensional simplices not contained in hyperplanes through the origin such that $\dim S_i\cap S_j<n-1$ if $i\neq j$. Valuations on $\calP^n$ satisfy the inclusion-exclusion principle (cf. \cite{Klain:Rota}, p.\ 7) and, therefore, 
$$\mu(P)=\sum_I (-1)^{|I|-1}\mu(S_I), $$
where the sum is taken over all ordered $k$-tuples $I =
(i_1,\ldots,i_k)$ such that $1 \leq i_1 < \ldots < i_k \leq n$
and $k = 1, \ldots ,m$. Here $|I|=k$ denotes the length of $I$
and $S_I = S_{i_1} \cap \cdots \cap S_{i_k}$. Since $\mu$ vanishes on polytopes of dimension strictly less than $n-1$, we obtain
$$\mu(P)=\sum_{i=1}^m \mu(S_i)=0.$$ 
Thus $\mu$ vanishes on all polytopes of dimension strictly less than $n$. 

Let $P$ be an $n$-dimensional polytope containing the origin. We can dissect $P=S_1\cup\cdots\cup S_m$ into simplices $S_i$, $i=1,\ldots,m$, having one vertex at origin such that $\dim S_i\cap S_j<n$ if $i\neq j$. As before, using the inclusion-exclusion principle, we obtain
$$\mu(P)=\sum_I (-1)^{|I|-1}\mu(S_I)=\sum_{i=1}^m \mu(S_i)=0.$$

Suppose now that $P$ is an $n$-dimensional polytope not containing the origin. Recall that a face $F$ of $P$ is called visible if $\{\lambda x: 0\leq\lambda<1\}\cap
P=\emptyset$ for every $x\in F$. Suppose that $P$ has exactly one
visible facet $F$. Since $\mu$ is a valuation, we have
$$\mu(P)+\mu( F_o)= \mu(P_o)+ \mu (F).$$
Since $\mu$ vanishes on $(n-1)$-dimensional polytopes and polytopes containing the origin we conclude that $\mu(P)=0$. 

Let $P$ be an $n$-dimensional polytope not containing the origin and having $m>1$ visible facets $F_1,\ldots,F_m$. For $i=1,\ldots, m$, put
$$C_i= P\cap \bigcup_{t\geq 0} tF_i.$$
It is easy to see that $C_i$ is an $n$-dimensional polytope having exactly one visible facet, $P=C_1\cup\cdots\cup C_m$ and $\dim C_i\cap C_j<n$ for $i\neq j$. An application of the inclusion-exclusion principle shows that 
$$\mu(P)=\sum_{i=1}^m \mu(C_i)=0.$$
Thus, $\mu=0$. 

For each $x\in\RR^n$, 
$$\mu_x(P)=h(\Phi_1P,x)-h(\Phi_2P,x)$$
defines a real valued valuation and from the first part of the proof we know that $\mu_x=0$. Since a convex body is uniquely determined by its support function, we obtain $\Phi_1 P=\Phi_2 P$ for each $P\in\calP^n$.

\end{proof}

Our next result reduces the proof of Theorem \ref{thm:2} to Minkowski valuations which are homogeneous of degree $1$ or $n+1$.

\begin{proposition}
Suppose $\Phi:\mathcal{K}^n\rightarrow\mathcal{K}^n$ is an $SL(n)$ equivariant continuous Minkowski valuation which is 
homogeneous of degree $r\neq 1, n+1$. Then 
$$\Phi K=\{0\}$$
for every $K\in\mathcal{K}^n$.
\end{proposition}
\begin{proof}
Let $S$ be the $(n-1)$-dimensional simplex with vertices $\{e_1,\ldots,e_n\}$. 
For $0<\lambda<1$ and integers $1\leq i<j\leq n$, we denote by 
$H_\lambda=H_\lambda(i,j)$ the hyperplane through the origin with normal vector $\lambda e_i - (1-\lambda) e_j$. Furthermore define linear 
maps $\phi_\lambda=\phi_\lambda(i,j)$ and $\psi_\lambda=\psi_\lambda(i,j)$ by
$$\phi_\lambda e_i= \lambda e_i +(1-\lambda)e_j, \ \ \phi_\lambda e_k= e_k\ \text{ for } k\neq i,$$
$$\psi_\lambda e_j= \lambda e_i +(1-\lambda)e_j, \ \ \phi_\lambda e_k= e_k\ \text{ for } k\neq j.$$
Observe that the hyperplane $H_\lambda$ splits the simplex $S$ into two simplices $\phi_\lambda S$ and $\psi_\lambda S$. Since $\Phi$ is a Minkowski valuation, we obtain 
$$\Phi S+\Phi (S\cap H_\lambda)=\Phi(\phi_\lambda S)+\Phi(\psi_\lambda S).$$
Using the fact that $\Phi$ is $SL(n)$ equivariant and homogeneous of degree $r\neq 1$, Lemma \ref{lem:1} yields $\Phi(S\cap H_\lambda)=\{0\}$.
Together with (\ref{eq:1}), we obtain
\begin{equation}\label{eq:3}
 \Phi S=\lambda^q \phi_\lambda\Phi S+(1-\lambda)^q \psi_\lambda\Phi S,
\end{equation}
where $q=(r-1)/n$.

Now let $1\leq k \leq n$ and choose $1 \leq i < j \leq n$ such
that $k \neq i,j$. This is possible since $n \geq 3$. By
(\ref{supptrans}) and (\ref{eq:3}), we have
$$h(\Phi S,e_k)=\lambda^q h(\Phi S,e_k)+(1-\lambda)^q  h(\Phi
S,e_k)$$
for every $0 < \lambda < 1$. Since $\Phi$ is
homogeneous of degree $r \neq n+1$, we have $q \neq 1$, which
implies $h(\Phi S,e_k)=0$. Similarly, we obtain $h(\Phi
S,-e_k)=0$. Since this holds for every $1\leq k \leq n$, we must
have 
\begin{equation}\label{eq:null}
 	 \Phi S=\{0\}.
 \end{equation}

Since $\Phi$ is $SL(n)$ equivariant, we deduce from (\ref{eq:null}) that $\Phi$ vanishes on every $(n-1)$-dimensional simplex not contained in a hyperplane through the origin. From Theorem \ref{thm:3} we know in particular that $\Phi$ vanishes on $n$-dimensional simplices having one vertex at the origin. Furthermore we have $\Phi\{0\}=\{0\}$ by the $SL(n)$ equivariance of $\Phi$. Applying Lemma \ref{lem:eq} shows that $\Phi P=\{0\}$
for every polytope $P$. By continuity of $\Phi$, we conclude that $\Phi K=\{0\}$ for every $K\in\calK^n$. 
\end{proof}

\section{Proof of the main result in the case $r=1$}

In this section we shall prove Theorem \ref{thm:2} in the case that $\Phi$ is homogeneous of degree $r=1$. In fact, we prove 
a slightly stronger result, since every $GL(n)$ equivariant Minkowski valuations is in particular $SL(n)$ equivariant and homogeneous. 

As the next lemma shows, we can reduce the dimension of the problem drastically and, as a consequence, give a simple and direct proof of Theorem \ref{thm:2} under the assumption that the degree of homogeneity is $1$.

Denote by $\pi_x$, $x\in\RR^n\setminus\{0\}$, the orthogonal projection onto $\spn\{x\}$. Note that an $SL(n)$ equivariant continuous Minkowski valuation which is homogeneous of degree $r=1$ is already determined by its values on convex bodies contained in lines through the orgin, 
\begin{equation}\label{eq:5}
 h(\Phi K,u)=h(\Phi(\pi_u K),u),\qquad u\in S^{n-1}.
\end{equation}
Indeed, for $s>0$ set
$$\phi_s=\begin{pmatrix}
   s I_{n-1} &0\\ 0 & 1
  \end{pmatrix}, $$
where, as before, $I_{n-1}$ is the $(n-1)\times(n-1)$ identity matrix. Since  $\det \phi_s>0$, using (\ref{eq:1}) gives
$$\Phi(\phi_s K)=\phi_s\Phi K.$$
Hence, by the continuity of $\Phi$, letting $s\rightarrow 0$ we obtain
$$\Phi(\pi_{e_n} K)=\pi_{e_n}(\Phi K).$$
Since $h(\pi_{e_n}L,e_n)=h(L,e_n)$ for every $L\in\calK^n$, we arrive at
$$h(\Phi K,e_n)=h(\Phi(\pi_{e_n}K),e_n).$$
Now let $u\in S^{n-1}$ and choose $\rho\in SO(n)$ such that $u=\rho e_n$. Since $\Phi$ is $SL(n)$ equivariant, using (\ref{supptrans})
and $\pi_{e_n}\circ \rho^{-1}=\rho^{-1}\circ \pi_u$, we see that
$$ h(\Phi K,u)=h(\pi_{e_n}\Phi(\rho^{-1} K), e_n)=h(\Phi(\rho^{-1}\pi_u K), e_n)=h(\Phi(\pi_u K),u)$$
for every $u\in S^{n-1}$.

For $x,y\in\RR^n$, we denote by $[x,y]$ the convex hull of $x$ and $y$.

\begin{lemma}\label{lem:2}
Let $\Phi_1,\Phi_2:\mathcal{K}^n\rightarrow\mathcal{K}^n$ be $SL(n)$ equivariant continuous Minkowski valuations which are homogeneous of degree $r=1$.
If there exists $x\in\RR^n$, $x\neq 0$, such that
\begin{equation}\label{eq:4}\Phi_1\{x\}=\Phi_2\{x\}\qquad \text{and}\qquad \Phi_1[0,x]=\Phi_2[0,x],\end{equation}
then 
$$\Phi_1=\Phi_2.$$
\end{lemma}
\begin{proof}
If (\ref{eq:4}) holds for one $x\neq 0$, then by $SL(n)$ equivariance (\ref{eq:4}) holds for every $x\in\RR^n\setminus\{0\}$. In order to prove $\Phi_1=\Phi_2$, by (\ref{eq:5}) it is sufficient to show that 
\begin{equation}\label{eq:6}
 \Phi_1[a u,bu]=\Phi_2[au,bu]
\end{equation}
for every $u\in S^{n-1}$ and $0\leq a\leq b$. To this end,
observe that by homogeneity and additivity 
$$ a \Phi_1[0,u]+\Phi_1[au,bu]=a\Phi_1\{u\} +b \Phi_1[0,u].$$
Hence, 
$$a\Phi_1[0,u]+\Phi_1[au,bu]=a\Phi_2[0,u]+\Phi_2[au,bu],$$
which implies (\ref{eq:6}) and proves the lemma.
\end{proof}

\begin{theorem}
A map $\Phi:\calK^n\rightarrow\calK^n$ is an $SL(n)$ equivariant continuous Minkowski valuation which is homogeneous of degree $1$ if and only if there are constants $a_1, a_2, a_3, a_4\geq0$ such that
$$\Phi K=a_1K+a_2(-K)+a_3K_o+a_4(-K_o)$$
for every $K\in\calK^n$.
\end{theorem}
\begin{proof}
We define the numbers $z_1=h(\Phi[0,e_1],-e_1)$, $z_2=h(\Phi[0,e_1],e_1)$, $z_3=h(\Phi\{e_1\},-e_1)$ and 
$z_4=h(\Phi\{e_1\},e_1)$.  Let $S$ be the convex hull of $e_1$ and $e_2$. Using (\ref{eq:5}) we find that
\begin{equation}\label{eq:7}
  h(\Phi S, e_1+e_2)=z_4\qquad\text{and}\qquad h(\Phi S, -e_1+e_2)=z_1+z_2.
\end{equation}
In fact, since $\pi_x S=\frac{1}{2}\{x\}$, $x=e_1+e_2$, and $\pi_y S=\frac{1}{2}[-y,y]$, $y=-e_1+e_2$, we obtain 
$$h(\Phi S,x)=\frac{1}{2}h(\Phi\{x\},x)=h(\Phi\{e_1\},e_1)=z_2$$
and 
$$h(\Phi S,y)=\frac{1}{2}h(\Phi[0,y],y)+\frac{1}{2}h(\Phi[0,y],-y)=z_1+z_2.$$
Since support functions are subadditive, we have

$$2h(\Phi S,e_2)\leq h(\Phi S,e_1+e_2)+h(\Phi S,-e_1+e_1),$$
which together with (\ref{eq:7}) implies
\begin{equation} \label{eq:8}
	2z_2\leq z_4+z_1+z_2 .
\end{equation}
Similarly, 
$$ 2h(\Phi S,-e_1)\leq h(\Phi S, -e_1-e_2)+h(\Phi S,-e_1+e_2)$$
gives 
\begin{equation} 
2z_1\leq z_3+z_1+z_2.
\end{equation}

Let $T$ be the convex hull of $e_1$, $e_2$ and $e_1+e_2$. A simple calculation using (\ref{eq:5}) shows that $h(\Phi T,e_1+e_2)=h(\Phi[e_1, 2e_1],e_1)$. In fact, for $x=e_1+e_2$ we have $\pi_x T=\frac{1}{2}[x,2x]$ and therefore $h(\Phi T,x)=\frac{1}{2} h(\Phi[x,2x],x)=h(\Phi[e_1, 2e_1],e_1)$. Since $\Phi$ is a Minkowski valuation and homogeneous of degree $1$, we have
$$\Phi[0,e_1]+\Phi[e_1,2e_1]=\Phi\{e_1\}+\Phi [0,2e_2]=\Phi\{e_1\}+2\Phi[0,e_1],$$
hence,
$$\Phi[e_1,2e_1]=\Phi\{e_1\}+\Phi[0,e_1].$$
We deduce that $h(\Phi T,e_1+e_2)=z_2+z_4$. 
Using again the subadditivity of support functions, we have
$$h(\Phi T,e_1+e_2)\leq h(\Phi T,e_1)+h(\Phi T,e_2).$$
Thus,
\begin{equation}\label{eq:10}
 	z_4\leq z_2.
 \end{equation} 
Similar reasoning, using $h(\Phi T,-e_1-e_2)\leq h(\Phi T,-e_1)+h(\Phi T,-e_2)$, gives
\begin{equation}\label{eq:11}
 	z_3\leq z_1.
 \end{equation} 
We remark that the inequalities (\ref{eq:10}) and (\ref{eq:11}) imply that for any $SL(n)$ equivariant continuous Minkowski valuation 
$$\Phi\{x\}\subset\Phi[0,x].$$ 

We define numbers
\begin{align*}
a_1&=z_1-z_3,\\
a_2&=z_2-z_4,\\
a_3&=z_2+z_3-z_1,\\
a_4&=z_1+z_4-z_2,
\end{align*}
which are all nonnegative by (\ref{eq:8})-(\ref{eq:11}). Let $\Psi$ be the $SL(n)$ equivariant Minkowski valuation given
by
$$\Psi K=a_1K+a_2(-K)+a_3K_o+a_4(-K_o).$$
A simple computation shows that 
$$\Phi\{e_1\}=\Psi\{e_1\}\qquad\text{and}\qquad \Phi[0,e_1]=\Psi[0,e_1],$$
which by Lemma \ref{lem:2} implies $\Phi=\Psi$ and proves the theorem.
\end{proof}

\section{Proof of the main result in the case $r=n+1$}

In this section we complete the proof of Theorem \ref{thm:2}. In the following, unless otherwise specified, $\Phi:\calK^n\rightarrow\calK^n$ will always denote an $SL(n)$ equivariant continuous Minkowski valuation which is homogeneous of degree $r=n+1$. The following technical lemma can be found in \cite{Ludwig:Minkowski}*{Lemma 3}.
\begin{lemma} \label{lem:3}
Suppose $s\in\mathbb{R}$ and $f:\mathbb{R}^2\rightarrow\mathbb{R}$ is a function which is positively homogeneous of degree $r$ and which satisfies
$$f(x)=\lambda^s f(A^t_\lambda x)+(1-\lambda)^s f(B^t_\lambda x)\ \text{ for }\ 0<\lambda<1,\ x\in\RR^2,$$
where 
$$A_\lambda = \begin{pmatrix}	\lambda & 0 \\ 1-\lambda & 1\\ \end{pmatrix}\ \text{ and }\ B_\lambda = \begin{pmatrix}	1 & \lambda \\ 0 & 1-\lambda\\ \end{pmatrix}.$$
Then for $x_1> x_2\geq0$
$$f(x_1,x_2)=\frac{x_1^{s+r}-x_2^{s+r}}{(x_1-x_2)^{s}} f(1,0),$$
$$f(-x_1,-x_2)=\frac{x_1^{s+r}-x_2^{s+r}}{(x_1-x_2)^{s}} f(-1,0),$$
for $x_2> x_1\geq 0$
$$f(x_1,x_2)=\frac{x_2^{s+r}-x_1^{s+r}}{(x_2-x_1)^s}f(0,1),$$
$$f(-x_1,-x_2)=\frac{x_2^{s+r}-x_1^{s+r}}{(x_2-x_1)^s}f(0,-1),$$
and for $x_1,x_2\geq0$, $x_1+x_2>0$
$$f(-x_1,x_2)=\frac{x_2^{s+r}}{(x_1+x_2)^s}f(0,1)+\frac{x_1^{s+r}}{(x_1+x_2)^s}f(-1,0),$$
$$f(x_1,-x_2)=\frac{x_1^{s+r}}{(x_1+x_2)^s}f(1,0)+\frac{x_2^{s+r}}{(x_1+x_2)^s}f(0,-1).$$
\end{lemma}

Fix two integers $i$ and $j$, $1\leq i<j\leq n$. We define a function $f:\RR^2\rightarrow \RR$ by $f(x_1,x_2)=h(\Phi S,x_1e_i+x_2e_j)$, where $S$ is the $(n-1)$-dimensional simplex with vertices $\{e_1,\ldots,e_n\}$. It follows from (\ref{eq:3}) that $f$ satisfies the assumptions of Lemma \ref{lem:3}. Observe that from (\ref{invar}) we have $f(1,0)=f(0,1)$ and $f(-1,0)=f(0,-1)$. Thus, the equations obtained from Lemma \ref{lem:3} simplify to
\begin{align*}
f(x_1,x_2)&=(x_1+x_2) f(1,0),\\
f(-x_1,-x_2)&=(x_1+x_2)f(-1,0),
\end{align*} 
and
$$(x_1+x_2)f(-x_1,x_2)=x_2^{2}f(1,0)+x_1^2f(-1,0),$$
for $ x_1, x_2\geq 0$. In particular, $f(x_1,x_2)=f(x_2,x_1)$ for $x_1,x_2\in\RR$.

Similarly, we define $g:\RR^2\rightarrow \RR$ by $g(x_1,x_2)=h(c_1 m(T)+c_2 \rmM T,x_1e_i+x_2e_j)$, where $T$ is the $n$-dimensional simplex with vertices $\{0,e_1,\ldots,e_n\}$ and $c_1,c_2\in\RR$, $c_2\geq0$. Applying Lemma \ref{lem:3} and using (\ref{invar}), we get as before
\begin{align*}
g(x_1,x_2)&=(x_1+x_2) g(1,0),\\
g(-x_1,-x_2)&=(x_1+x_2)g(-1,0),
\end{align*} 
and
$$	(x_1+x_2)g(-x_1,x_2)=x_2^{2}g(1,0)+x_1^2g(-1,0),$$
for $ x_1, x_2\geq 0$. Since $m(T)$ is a point and $\rmM T=-\rmM T$, we see that
\begin{align*}
	g(1,0)&=c_1h(m(T),e_1)+c_2h(\rmM T,e_1),\\
	g(-1,0)&=-c_1h(m(T),e_1)+c_2h(\rmM T,e_1).
\end{align*}
By the subadditivity of support functions, $f(1,0)+f(-1,0)\geq 0$. Since  $h(\rmM T,e_1)\geq 0$, it is possible to choose
$c_1\in \RR$ and $c_2\geq0$ such that $g(1,0)=f(1,0)$ and $g(-1,0)=f(-1,0)$. We conclude that $f=g$. Notice that by (\ref{invar}) the constants $c_1$ and $c_2$ do not depend on the particular choice of $i$ and $j$.  On the level of support functions this means that
\begin{equation}\label{eq:13}
 	h(\Phi S,x_1e_i+x_2e_j)=h(c_1 m(T)+c_2 \rmM T,x_1e_i+x_2e_j)
 \end{equation}
for any choice of integers $1\leq i<j\leq n$ and $x_1,x_2\in\RR$.
In fact, even more is true. The next lemma (cf. \cite{Ludwig:Minkowski}*{Lemma 4}) shows that (\ref{eq:3}) and (\ref{eq:13}) imply 

\begin{equation}
 	 \Phi S= c_1m(T)+c_2 \rmM T.
\end{equation}
For any integers $1\leq i<j\leq n$ and $\lambda\in(0,1)$ we define linear maps $\phi_\lambda=\phi_\lambda(i,j)$ and $\psi_\lambda=\psi_\lambda(i,j)$ by
$$\phi_\lambda e_i= \lambda e_i +(1-\lambda)e_j, \ \ \phi_\lambda e_k= e_k\ \text{ for } k\neq i$$
and
$$\psi_\lambda e_j= \lambda e_i +(1-\lambda)e_j, \ \ \phi_\lambda e_k= e_k\ \text{ for } k\neq j.$$
\begin{lemma}
Let $q\in\mathbb{R}$ and $h:\mathbb{R}^n\rightarrow\mathbb{R}$, $n\geq 3$, be a function which satisfies
$$h(x)=\lambda^q h(\phi^{t}_\lambda x)+(1-\lambda)^q h(\psi^{t}_\lambda x)$$
for any $\ 0<\lambda<1$, integers $1\leq i<j\leq n$, and $x\in\mathbb{R}^n.$ If $h(x)=0$ for every $x\in\mathbb{R}^n$ where at most two coordinates are not zero, then $h=0$.
\end{lemma}

We are now ready to complete the proof of Theorem \ref{thm:2} treating the only remaining case $r=n+1$.

\begin{theorem}
A map $\Phi:\calK^n\rightarrow\calK^n$ is an $SL(n)$ equivariant continuous Minkowski valuation which is homogeneous of degree $r=n+1$ if and only if there are constants $a_1,a_2\in\mathbb{R}$ and $a_3,a_4\geq 0$ such that  
$$ZK=a_1 m(K)+a_2 m_*(K)+ a_3\rmM K+ a_4 \rmM_* K$$
for every $K\in\mathcal{K}^n$.
\end{theorem} 
\begin{proof}
We know from Theorem \ref{thm:3} that there are constants $a_1\in\RR$ and $a_3\geq 0$ such that 
$$\Phi K=a_1 m(K) +a_3 \rmM K$$ 
for every $K\in\calK^n$ containing the origin. We define an $SL(n)$ equivariant continuous Minkowski valuation $\Psi$ by 
$$\Psi K=a_1 m(K)+a_2 m_*(K)+ a_3\rmM K+ a_4 \rmM_* K,$$
where $a_2=c_1$ and $a_4=c_2$ are constants from (\ref{eq:13}). An immediate consequence of this definition is that
$$\Phi T=\Psi T$$
for every $n$-dimensional simplex $T$ having one vertex at the origin and that 
\begin{equation}\label{eq:14}
	\Phi S=\Psi S,
\end{equation}
where $S$ is the $(n-1)$-dimensional simplex with vertices $\{e_1,\ldots, e_n\}$. Since both $\Phi$ and $\Psi$ are $SL(n)$ equivariant, we deduce that (\ref{eq:14}) holds true for every $(n-1)$-dimensional simplex not contained in a hyperplane through the origin. Therefore, Lemma \ref{lem:eq} yields $\Phi=\Psi$ and proves the theorem.

\end{proof}

\subsection*{Acknowledgment}{The work of the author was
supported by the  Austrian Science Fund (FWF), within
the project ``Minkowski valuations and geometric inequalities",
Project number: P\,22388-N13.}

\end{document}